\documentclass[12pt,oneside,english]{amsart}
\usepackage[T1]{fontenc}
\usepackage[latin9]{inputenc}
\usepackage{geometry}
\usepackage{babel}
\usepackage{amstext}
\usepackage{amsthm}
\usepackage{esint}
\usepackage[unicode=true,pdfusetitle,
 bookmarks=true,bookmarksnumbered=false,bookmarksopen=false,
 breaklinks=false,pdfborder={0 0 1},backref=false,colorlinks=false]
 {hyperref}

\makeatletter
\numberwithin{equation}{section}
\numberwithin{figure}{section}
\theoremstyle{plain}
\newtheorem{thm}{\protect\theoremname}[section]
\theoremstyle{plain}
\newtheorem{lem}{\protect\lemmaname}[section]
\theoremstyle{plain}
\newtheorem{prop}{\protect\propositionname}[section]
\theoremstyle{plain}
\newtheorem{cor}{\protect\corollaryname}[section]

\usepackage{babel}

\providecommand{\lemmaname}{Lemma}
\providecommand{\theoremname}{Theorem}

\makeatother

\providecommand{\corollaryname}{Corollary}
\providecommand{\lemmaname}{Lemma}
\providecommand{\propositionname}{Proposition}
\providecommand{\theoremname}{Theorem}

\begin{document}
\title[Weighted Partial Sums of a Random Multiplicative Function ]{Weighted partial sums of a random multiplicative function and their
positivity }
\author{Shuming Liu}
\address{School of Mathematics and Statistics, HNP-LAMA, Central South University
\\
 Changsha 410083, Hunan, People's Republic of China}
\email{232111040@csu.edu.cn}
\author{Bing He}
\address{School of Mathematics and Statistics, HNP-LAMA, Central South University
\\
 Changsha 410083, Hunan, People's Republic of China}
\email{yuhelingyun@foxmail.com; yuhe123456@foxmail.com }
\begin{abstract}
In this paper, we study the probability that some weighted partial
sums of a random multiplicative function $f$ are positive. Applying
the characteristic decomposition, we obtain that if $S$ is a non-empty
subset of the multiplicative residue class group $\ensuremath{(\mathbb{Z}/m\mathbb{Z})^{\times}}$
with $m$ being a fixed positive integer and $\ensuremath{A=\{a+mn\mid n=0,1,2,3,\cdots\}}$
with $a\in S,$ then there exists a positive number $\delta$ independent
of $\ensuremath{x},$ such that
\[
\mathbb{P}\left(\sum_{A\cap[1,x)}\frac{f(n)}{n}<0\right)>\delta
\]
unless the coefficients of the real characters in the expansion of
the characteristic function of $S$ according to the characters of
$(\mathbb{Z}/m\mathbb{Z})^{\times}$ are all non-negative, and the
coefficients of the complex characters are all zero, in which case
we have
\[
\mathbb{P}\left(\sum_{A\cap[1,x)}\frac{f(n)}{n}<0\right)=O\left(\exp\left(-\exp\left(\frac{\ln x}{C\ln_{2}x}\right)\right)\right)
\]
for a positive constant $C.$ This includes as a special case a result
of Angelo and Xu. We also extend the result to the cyclotomic field
$\ensuremath{K_{n}=\mathbb{Q}(\zeta_{n})}$ with $\ensuremath{\zeta_{n}=e^{2\pi i/n}}$
and study the probability that these generalized weighted sums are
positive.  In addition, we deal with the positivity problem of certain
partial sums related to the celebrated Ramanujan tau function $\tau(n)$
and the Ramanujan modular form $\Delta(q),$ and obtain an upper bound
for the probability that these partial sums are negative in a more
general situation.
\end{abstract}

\thanks{The second author is the corresponding author.}

\maketitle

\section{Introduction}

In their recent paper \cite{key-1}, motivated by a Tur\'{a}n  conjecture
on the positivity of the weighted partial sums of the Liouville function
$\lambda$ and its connections with the Riemann hypothesis, Angelo
and Xu proved that the probability that the partial sum

\[
\sum_{n\leq x}\frac{f(n)}{n}
\]
is negative for a fixed large $x$ is at most $O\left(\exp\left(-\exp\left(\frac{\ln x}{C\ln_{2}x}\right)\right)\right)$,
where $f$ is a random completely multiplicative function and $C$
is a positive constant. The random completely multiplicative function
is defined to be a funcion $f$ such that $f(p)=\pm1$ with probabilities
$1/2$ independently at each prime, and it can be extended completely
multiplicatively to all narural numbers. They prove it approximating
the above partial sum by large Euler product, and using the Rankin
trick to estimate the small tails. This bound was later improved by
Kerr and Klurman \cite{key-2} using a different method. Another related
topic concerns the sign changes and the lower bound of the partial
sums 
\[
\ensuremath{\sum_{n\leq x}\frac{f(n)}{n^{\sigma}}},
\]
 where $\ensuremath{0\leq\sigma<1}.$ Related work can be found in
\cite{A,key-7,key-6}.

We now consider the partial sum 
\[
\sum_{\substack{n\equiv1(\bmod4)\\
n\leq x
}
}\frac{f(n)}{n},
\]
 which is similar to the sum

\[
\sum_{n\leq x}\frac{f(n)}{n}.
\]
Using the characteristic decomposition and combining the results of
Angelo and Xu, the probability estimate

\[
\mathbb{P}\left(\sum_{\substack{n\equiv1(\bmod4)\\
n\leq x
}
}\frac{f(n)}{n}<0\right)=O\left(\exp\left(-\exp\left(\frac{\ln x}{C\ln_{2}x}\right)\right)\right)
\]
still holds similarly.

The first objective of this paper is to extend the above result to
more general congruence classes. Applying the characteristic decomposition,
we obtain the following conclusion, which includes as a special case
the result \cite[Theorem 1.2]{key-1} of Angelo and Xu.
\begin{thm}
\label{Theorem 1.2} For a fixed positive integer $\ensuremath{m},$
suppose $S$ is a non-empty subset of the multiplicative residue class
group $\ensuremath{(\mathbb{Z}/m\mathbb{Z})^{\times}}.$ Define $\ensuremath{A=\{a+mn\mid n=0,1,2,3,\cdots\}}$
with $a\in S.$ Then there exists a positive number $\delta$ independent
of $\ensuremath{x},$ such that

\[
\mathbb{P}\left(\sum_{A\cap[1,x)}\frac{f(n)}{n}<0\right)>\delta
\]
unless the coefficients of the real characters in the expansion of
the characteristic function of $S$ according to the characters of
$(\mathbb{Z}/m\mathbb{Z})^{\times}$ are all non-negative, and the
coefficients of the complex characters are all zero, in which case
we have

\begin{equation}
\mathbb{P}\left(\sum_{A\cap[1,x)}\frac{f(n)}{n}<0\right)=O\left(\exp\left(-\exp\left(\frac{\ln x}{C\ln_{2}x}\right)\right)\right),\label{eq:1-1}
\end{equation}
where $C$ is a positive constant.
\end{thm}
When $a=m=1,$ the probability asymptotic \eqref{eq:1-1} of Theorem
\ref{Theorem 1.2} reduces to the result \cite[Theorem 1.2]{key-1}
of Angelo and Xu.

We also extend the result to the $n$-th cyclotomic field $\ensuremath{K_{n}:=\mathbb{Q}(\zeta_{n})},$
where $\ensuremath{\zeta_{n}:=e^{2\pi i/n}}.$ Let $f$ be a function
defined on the ring of integral ideals of $K_{n}.$ At each prime
ideal, $f$ takes values from the set $\ensuremath{\{1,-1\}}$ independently
and with equal probability. For any integral ideal $\mathfrak{a}$
with a prime ideal factorization $\ensuremath{\mathfrak{a}=\mathfrak{p}_{1}^{e_{1}}\cdots\mathfrak{p}_{s}^{e_{s}}},$
the value of $f$ on $\mathfrak{a}$ satisfies $\ensuremath{f(\mathfrak{a})=f(\mathfrak{p}_{1})^{e_{1}}\cdots f(\mathfrak{p}_{s})^{e_{s}}}.$
This function $f$ is an extension of a random completely multiplicative
function defined on the set of positive integers to the ring of integral
ideals of $K_{n}.$ We now consider the partial sums

\[
S_{x,K_{n}}:=\sum_{N(\mathfrak{a})\leq x}\frac{f(\mathfrak{a})}{N(\mathfrak{a})}
\]
with $N(\mathfrak{a})$ being the norm of the integral ideal $\mathfrak{a},$
and investigate the probability that $S_{x,K_{n}}<0.$ By combining
the prime ideal decomposition theorem in $K_{n}$ and the Brun-Titchmarsh
inequality, we proved that when the degree $n$ of the cyclotomic
field is not excessively large compared to $\ensuremath{x},$ similar
probability estimates remain valid. More precisely, we obtain the
following result.
\begin{thm}
\label{Theorem 1.1}Assume that $K_{n}=\mathbb{Q}(\zeta_{n})$ is
the $n$-th cyclotomic field, and $n<(\log x)^{A}$, where $A$ is
a positive number. Then

\[
\mathbb{P}(S_{x,K_{n}}<0)=O\left(\exp\left(-\exp\left(\frac{\ln x}{C(A)\ln_{2}x}\right)\right)\right)
\]
where $C(A)$ is a positive constant depending on $A$.
\end{thm}

The final objective of this paper is to handle the positivity problem
of certain partial sums related to the Ramanujan tau function give
by 
\[
\sum_{n=1}^{\infty}\tau(n)q^{n}:=q\prod_{k\geq1}(1-q^{k})^{24},\quad|q|<1.
\]
Suppose that $\Delta(q):=q\prod_{k\geq1}(1-q^{k})^{24}$ is the Ramanujan
modular form, and $L(\Delta,s)$ is the associated $L$-function.
The result of Mordell tells us that $\tau(n)$ is multiplicative and
satisfies the equation $\tau(p^{k+1})=\tau(p)\tau(p^{k})-p^{11}\tau(p^{k-1}),k\geq1.$
Deligne's extraordinary work shows that the equation $\ensuremath{\tau(p)=2p^{11/2}\cos\theta_{p}}$
holds for some $\ensuremath{\theta_{p}\in(0,\pi)}.$ The famous Sato-Tate
conjecture suggests that the values of $\theta_{p}$ follow the distribution
model:

\[
\mathbb{P}(\alpha<\theta_{p}<\beta)=\frac{2}{\pi}\int_{\alpha}^{\beta}\sin^{2}\theta d\theta.
\]
Related results can be found in \cite[Chapter 3]{key-4}.

It is natural to consider the following probability model. Suppose
$\varrho(n)$ is a multiplicative function, and it satisfies the recurrence
relation:

\[
\varrho(p^{k+1})=\varrho(p)\varrho(p^{k})-p^{11}\varrho(p^{k-1}),\quad k\geq1.
\]
Additionally, we assume that $\ensuremath{\varrho(p)=2p^{11/2}\cos\theta_{p}},$
where $\theta_{p}$ is a family of independent and identically distributed
random variables, taking values in $(0,\pi)$ and following distribution

\[
\mathbb{P}(\alpha<\theta_{p}<\beta)=\frac{2}{\pi}\int_{\alpha}^{\beta}\sin^{2}\theta d\theta.
\]

For the partial sum 
\[
\ensuremath{I_{x}:=\sum_{n\leq x}\frac{\varrho(n)}{n^{13/2}}},
\]
we are concerned with the upper bound of the probability $\ensuremath{\mathbb{P}(I_{x}<0)}.$
We extend this result to obtain an upper bound for this probability.
In fact, we arrive at a conclusion in a more general situation.
\begin{thm}
\label{Theorem 1.3} Suppose $\varrho(n)$ is a multiplicative function
satisfying the recurrence relation:

\[
\varrho(p^{k+1})=\varrho(p)\varrho(p^{k})-p^{m}\varrho(p^{k-1}),\quad k\geq1,
\]
where $m\geq0$ is a fixed integer. Additionally, we suppose that
$\ensuremath{\varrho(p)=2p^{m/2}\cos\theta_{p}},$ where $\theta_{p}\in(0,\pi)$
is a family of independent and identically distributed random variables
and satisfies

\[
\mathbb{P}(\alpha<\theta_{p}<\beta)=\frac{2}{\pi}\int_{\alpha}^{\beta}\sin^{2}\theta d\theta.
\]
Then we have
\[
\mathbb{P}(I_{x}^{(m)}<0)=O\left(\exp\left(-\exp\left(\frac{\ln x}{C\ln_{2}x}\right)\right)\right),
\]
where $C$ is a positive constant depending on $m$ and
\[
I_{x}^{(m)}:=\sum_{n\leq x}\frac{\varrho(n)}{n^{(m+2)/2}}.
\]
\end{thm}
\subsection*{Notation.} We write $f\ll g$ or $f=O(g)$ if there
exists a positive constant $C$ such that $f\leq Cg.$ Similarly,
we write $f\ll_{A}g$ or $f=O_{A}(g)$ when the constant $C$ depends
on the parameter $A.$ 

\section{Auxiliary results}

In this section, we list some auxiliary results that we will use to
prove Theorems \ref{Theorem 1.1}, \ref{Theorem 1.2} and \ref{Theorem 1.3}.

We begin this section with an important result in \cite[Chapter 7, Theorem 7.3.1]{key-5}.
\begin{lem}
\label{Lemma 2.2} Suppose that $1\leq l\leq k<y\leq x,(k,l)=1.$
Then we have

\[
\pi(x;k,l)-\pi(x-y;k,l)<\frac{3y}{\varphi(k)\ln(y/k)},
\]
where $\varphi(k)$ denotes the Euler totient function and
\[
\pi(x;k,l):=\textrm{\ensuremath{\sharp}}\left\{ p\leq x:p\:is\:prime,p\equiv l(\bmod\:k)\right\} .
\]
In particular,

\[
\pi(x;k,l)<\frac{3x}{\varphi(k)\ln(x/k)}.
\]
\end{lem}
The next lemma concerns the decomposition of rational prime numbers
in cyclotomic fields.
\begin{lem}
\label{Lemma 2.1} \textup{(See \cite[Chapter 1, Proposition 10.2]{key-3})}
Let $n=\prod_{p}p^{v_{p}}$be the prime factorization of $n$ and,
for every prime number $p,$ let $f_{p}$ be the smallest positive
integer such that

\[
p^{f_{p}}\equiv1\left(\bmod\:n/p^{v_{p}}\right).
\]
Then, in $Q(\zeta_{n})$ one has the factorization

\[
p=\left(\mathfrak{p}_{1}\cdots\mathfrak{p}_{r}\right)^{\varphi(p^{v_{p}})},
\]
where $\mathfrak{p}_{1},\cdots,\mathfrak{p}_{r}$ are distinct prime
ideals, all of degree $f_{p}.$
\end{lem}
\begin{lem}
\label{Lemma 2.4}Let $q$ be a fixed prime number, or $q=1.$ As
$x\rightarrow\infty,$ assume that

\[
k<x^{\frac{1}{D\ln_{2}x}}
\]
where $D$ is a positive constant. Then, we have the following asymptotic
estimate

\[
\sum_{\substack{P^{+}(u)\leq x\\
P^{-}(u)\geq q\\
u>x^{k}
}
}\frac{d_{2k}(u^{2})q^{\Omega(u)}}{u^{2}}\ll\exp\left(-\frac{Dk}{2}\ln_{2}x\right),
\]
where $P^{+}(u)$ and $P^{-}(u)$ denote the greatest and the smallest
prime factors of $u$ respectively, and $\Omega(u)$ represents the
total number of prime factors of $u$ counting multiplicities.
\end{lem}
\begin{proof}
Set

\begin{align*}
T_{1} & :=\sum_{\substack{P^{+}(u)\leq x\\
P^{-}(u)>q\\
u>x^{k}
}
}\frac{d_{2k}(u^{2})q^{\Omega(u)}}{u^{2}},\\
T_{2} & :=\sum_{\substack{P^{+}(u)\leq x\\
P^{-}(u)=q\\
u>x^{k}
}
}\frac{d_{2k}(u^{2})q^{\Omega(u)}}{u^{2}},
\end{align*}

We first estimate $T_{1}.$ Take $R=2k^{2/\sigma}$ with $\sigma=2-\frac{D\ln_{2}x}{\ln x}$
and assume that $x$ large enough such that $\sigma>1.99.$ Set

\begin{align*}
T_{1,1} & :=\prod_{q<p\leq R}\frac{(1-\frac{q}{p^{\sigma/2}})^{-2k}+(1+\frac{q}{p^{\sigma/2}})^{-2k}}{2},\\
T_{1,2} & :=\prod_{R<p<x}\frac{(1-\frac{q}{p^{\sigma/2}})^{-2k}+(1+\frac{q}{p^{\sigma/2}})^{-2k}}{2}.
\end{align*}
Then

\begin{align}
\begin{aligned}T_{1} & <\frac{1}{x^{k(2-\sigma)}}\sum_{\substack{P^{+}(u)\leq x\\
P^{-}(u)>q
}
}\frac{d_{2k}(u^{2})q^{\Omega(u)}}{u^{\sigma}}\\
 & =\frac{1}{x^{k(2-\sigma)}}\prod_{\substack{P^{+}(u)\leq x\\
P^{-}(u)>q
}
}\frac{(1-\frac{q}{p^{\sigma/2}})^{-2k}+(1+\frac{q}{p^{\sigma/2}})^{-2k}}{2}\\
 & =\frac{1}{x^{k(2-\sigma)}}T_{1,1}\cdot T_{1,2}.
\end{aligned}
\label{eq:4-1-1}
\end{align}
For $T_{1,1},$ we have

\begin{align*}
T_{1,1} & <\prod_{q<p\leq R}(1-\frac{q}{p^{\sigma/2}})^{-2k}=\exp(kO(\sum_{q<p\leq R}\frac{1}{p^{\sigma/2}}))\\
 & =\exp(kO(R^{1-\frac{\sigma}{2}}\ln_{2}R))=\exp(kO(\ln_{2}x)).
\end{align*}
For $T_{1,2},$ we have

\begin{align*}
T_{1,2} & <\prod_{R<p}(1-\frac{1}{p^{\sigma}})^{-(2k^{2}+2k)}=\exp(k^{2}O(\sum_{p>R}\frac{1}{p^{\sigma}}))\\
 & =\exp(k^{2}O(R^{1-\sigma}))=\exp(O(k)),
\end{align*}
where in the first step we have used \cite[eq.(2.9)]{key-1}. Substituting
the estimates for $T_{1,1}$ and $T_{1,2}$ into the right side of
\eqref{eq:4-1-1} we get

\[
T_{1}=\frac{1}{x^{k(2-\sigma)}}\exp(kO(\ln_{2}x)).
\]

Now we estimate $T_{2}.$ Set $u=q^{l}v,P^{-}(v)>q.$ Then

\[
T_{2}=\sum_{\substack{l=1}
}^{\infty}\frac{d_{2k}(q^{2l})}{q^{l}}\sum_{\substack{q<P(v)\leq x\\
v>x^{k}/q^{l}
}
}\frac{d_{2k}(v^{2})q^{\Omega(v)}}{v^{2}}
\]
Proceeding as in estimating $T_{1},$ we have

\begin{align*}
\sum_{\substack{q<P(v)\leq x\\
v>x^{k}/2^{l}
}
}\frac{d_{2k}(v^{2})q^{\Omega(v)}}{v^{2}} & <\frac{1}{(x^{k}/q^{l})^{2-\sigma}}\sum_{\substack{q<P(v)\leq x\\
v>x^{k}/q^{l}
}
}\frac{d_{2k}(v^{2})q^{\Omega(v)}}{v^{\sigma}}\\
 & <\frac{1}{(x^{k}/q^{l})^{2-\sigma}}\sum_{\substack{q<P(v)\leq x}
}\frac{d_{2k}(v^{2})q^{\Omega(v)}}{v^{\sigma}}\\
 & =\frac{1}{(x^{k}/q^{l})^{2-\sigma}}\prod_{q<p<x}\frac{(1-\frac{q}{p^{\sigma/2}})^{-2k}+(1+\frac{q}{p^{\sigma/2}})^{-2k}}{2}.
\end{align*}
Since

\[
\prod_{q<p<x}\frac{(1-\frac{q}{p^{\sigma/2}})^{-2k}+(1+\frac{q}{p^{\sigma/2}})^{-2k}}{2}=\exp(kO(\ln_{2}x)),
\]
we obtain

\[
T_{2}=\exp(kO(\ln_{2}x))\sum_{\substack{l=1}
}^{\infty}\frac{d_{2k}(q^{2l})}{q^{l}(x^{k}/q^{l})^{2-\sigma}}=\frac{\exp(kO(\ln_{2}x))}{x^{k(2-\sigma)}}\sum_{\substack{l=1}
}^{\infty}\frac{d_{2k}(q^{2l})}{q^{l(\sigma-1)}}
\]
Notice that

\[
\sum_{\substack{l=1}
}^{\infty}\frac{d_{2k}(q^{2l})}{q^{l(\sigma-1)}}<\left(1-\frac{1}{q^{\frac{\sigma-1}{2}}}\right)^{-2k}=\exp(O(k)).
\]
This implies that

\[
T_{2}=\frac{\exp(kO(\ln_{2}x))}{x^{k(2-\sigma)}}
\]

In view of the above we get

\[
\sum_{\substack{P^{+}(u)\leq x\\
P^{-}(u)\geq q\\
u>x^{k}
}
}\frac{d_{2k}(u^{2})q^{\Omega(u)}}{u^{2}}=T_{1}T_{2}\ll\frac{\exp(kO(\ln_{2}x))}{x^{k(2-\sigma)}}.
\]
This completes the proof.
\end{proof}
 The following lemma gives upper bounds for certain Euler products
over cyclotomic fields.
\begin{lem}
\label{Lemma 2.5} Suppose that

\[
2-\frac{D\ln_{2}x}{\ln x}<\sigma\leq2,\quad x>0,
\]
and as $x\rightarrow\infty,$ we have

\[
k<x^{\frac{1}{D\ln_{2}x}},\quad n<(\ln x)^{A}
\]
where $A,D$ are positive constants. Let

\[
Z_{1}:=\prod_{\substack{(\mathfrak{p},n)=1\\
f_{\mathfrak{p}}=1\\
N(\mathfrak{p})\leq x
}
}\left(\frac{\left(1+\frac{1}{N(\mathfrak{p})^{\sigma/2}}\right)^{-2k}+\left(1-\frac{1}{N(\mathfrak{p})^{\sigma/2}}\right)^{-2k}}{2}\right),
\]
\begin{align*}
Z_{2} & :=\prod_{\substack{\mathfrak{p}|n\\
N(\mathfrak{p})\leq x
}
}\left(\frac{\left(1+\frac{1}{N(\mathfrak{p})^{\sigma/2}}\right)^{-2k}+\left(1-\frac{1}{N(\mathfrak{p})^{\sigma/2}}\right)^{-2k}}{2}\right)
\end{align*}
and

\[
Z_{3}:=\prod_{\substack{(\mathfrak{p},n)=1\\
f_{\mathfrak{p}}>1\\
N(\mathfrak{p})\leq x
}
}\left(\frac{\left(1+\frac{1}{N(\mathfrak{p})^{\sigma/2}}\right)^{-2k}+\left(1-\frac{1}{N(\mathfrak{p})^{\sigma/2}}\right)^{-2k}}{2}\right).
\]
where $\mathfrak{p}$ denotes a prime ideal in the $n$-th cyclotomic
field $\ensuremath{K_{n}=Q(\zeta_{n})}$ and, $f_{\mathfrak{p}}$
and $N(\mathfrak{p})$ represents its inertia degree and norm respectively.
Then, for $j\in\{1,2,3\},$ we have

\[
Z_{j}<\exp(C_{j}(A)k\ln_{2}x)
\]
where $C_{j}(A)$ is a positive constant depending on $A.$
\end{lem}
\begin{proof}
We first consider $Z_{1}.$ Define $R=2k^{2/\sigma}.$ By using Lemma
\ref{Lemma 2.1}, we can rewrite $Z_{1}$ as 
\[
Z_{1}=\prod_{\substack{p\equiv1\bmod n\\
p\leq x
}
}\left(\frac{(1-\frac{1}{p^{\sigma/2}})^{-2k}+(1+\frac{1}{p^{\sigma/2}})^{-2k}}{2}\right)^{\varphi(n)}.
\]
We now decompose $Z_{1}$ into two parts:

\[
Z_{1}=Z_{1,1}Z_{1,2}
\]
where

\begin{align*}
Z_{1,1} & =\prod_{\substack{p\equiv1\bmod n\\
p\leq R
}
}\left(\frac{(1-\frac{1}{p^{\sigma/2}})^{-2k}+(1+\frac{1}{p^{\sigma/2}})^{-2k}}{2}\right)^{\varphi(n)},\\
Z_{1,2} & =\prod_{\substack{p\equiv1\bmod n\\
R\leq p\leq x
}
}\left(\frac{(1-\frac{1}{p^{\sigma/2}})^{-2k}+(1+\frac{1}{p^{\sigma/2}})^{-2k}}{2}\right)^{\varphi(n)}.
\end{align*}

For $Z_{1,1},$ we have

\begin{align*}
Z_{1,1} & \ll\prod_{\substack{p\equiv1\bmod n\\
p\leq R
}
}\left(1-\frac{1}{p^{\sigma/2}}\right)^{-2k\varphi(n)}=\exp\left(k\varphi(n)O\left(\sum_{\substack{p\equiv1\bmod n\\
p\leq R
}
}\frac{1}{p^{\sigma/2}}\right)\right).
\end{align*}
Applying Lemma \ref{Lemma 2.2} we get

\begin{align*}
\sum_{\substack{p\equiv1\bmod n\\
p\leq R
}
}\frac{\varphi(n)}{p^{\sigma/2}} & \ll\varphi(n)\sum_{R-1\geq m>n}\frac{\pi(m,n,1)}{m^{\sigma/2+1}}+\varphi(n)\frac{\pi(R;n,1)}{R^{\sigma/2}}\\
 & \ll\varphi(n)\sum_{\substack{m^{1/2}\leq n}
}\frac{\pi(m;n,1)}{m^{\sigma/2+1}}+\varphi(n)\sum_{\substack{n<m^{1/2}\leq R^{1/2}}
}\frac{\pi(m;n,1)}{m^{\sigma/2+1}}+\frac{R}{R^{\sigma/2}\ln R/n}\\
 & \ll\sum_{\substack{m^{1/2}\leq n}
}\frac{\varphi(n)}{m^{\sigma/2}n}+\sum_{\substack{n^{2}<m\leq R}
}\frac{1}{m^{\sigma/2}\ln m}+O_{A}(1)\\
 & \ll n^{2-\sigma}\ln n+R^{1-\sigma/2}\ln_{2}R\\
 & \ll_{A}\ln_{2}x.
\end{align*}
and so
\[
Z_{1,1}<\exp\left(k\varphi(n)O_{A}\left(\ln_{2}x\right)\right).
\]

It is easy to see that
\begin{align*}
Z_{1,2} & \ll\prod_{\substack{p\equiv1\bmod n\\
R\leq p\leq x
}
}\left(1-\frac{1}{p^{\sigma}}\right)^{-(2k^{2}+2k)\varphi(n)}\\
 & =\exp\left(k^{2}\varphi(n)O\left(\sum_{\substack{p\equiv1\bmod n\\
R\leq p\leq x
}
}\frac{1}{p^{\sigma}}\right)\right),
\end{align*}
where in the first step we have used \cite[eq.(2.9)]{key-1}. Notice
that

\[
k^{2}\varphi(n)\sum_{\substack{p\equiv1\bmod n\\
R\leq p
}
}\frac{1}{p^{\sigma}}\ll k^{2}\varphi(n)\sum_{R<m}\frac{1}{nm^{\sigma}}\ll k^{2}R^{1-\sigma}\ll R\ll ke^{D}.
\]
Then
\[
Z_{1,2}<\exp\left(O(ke^{D})\right).
\]

Combining the estimates for $Z_{1,1}$ and $Z_{1,2}$ gives 
\[
Z_{1}<\exp(C_{1}(A)k\ln_{2}x)
\]
 for a positive constant $C_{1}(A).$

We next consider $Z_{2}.$ For $Z_{2},$ applying Lemma \ref{Lemma 2.1},
we have

\begin{align*}
Z_{2} & <\prod_{\substack{l(\bmod n)\\
(l,n)=1,l>1
}
}\prod_{\substack{p\equiv l(\bmod n)\\
p^{f_{l}}\leq x
}
}\left(1-\frac{1}{p^{f_{l}\sigma/2}}\right)^{-2k\frac{\varphi(n)}{f_{l}}}\\
 & =\exp\left(O\left(k\sum_{\substack{1<l<n\\
(l,n)=1
}
}\sum_{\substack{p\equiv l(\bmod n)\\
p^{f_{l}}\leq x
}
}\frac{\varphi(n)}{f_{l}p^{f_{l}\sigma/2}}\right)\right).
\end{align*}
where $f_{l}$ represents the order of $l(\bmod n).$

By the inequlities $\sum_{m>n}\frac{1}{m^{f\sigma/2}}\ll\frac{1}{n^{f\sigma/2-1}},\pi(m,n,l)\ll\frac{m}{n},m>n$
and $l^{f_{l}}>n,f_{l}\geq2,$ we have

\begin{align*}
\sum_{\substack{p\equiv l(\bmod n)\\
p^{f_{l}}\leq x
}
}\frac{\varphi(n)}{f_{l}p^{f_{l}\sigma/2}} & \ll\frac{\varphi(n)}{f_{l}}\sum_{n=1}^{\infty}\frac{\pi(m,n,l)}{m^{f_{l}\sigma/2+1}}\\
 & <\frac{\varphi(n)}{f_{l}}\left(\frac{1}{l^{f_{l}\sigma/2+1}}+2\sum_{m>n}\frac{1}{nm^{f_{l}\sigma/2}}\right)\\
 & \ll\varphi(n)\left(\frac{1}{n^{\sigma/2}l}+\frac{1}{n^{\sigma}}\right).
\end{align*}
Thus

\begin{align*}
\sum_{\substack{1<l<n\\
(l,n)=1
}
}\sum_{\substack{p\equiv l(\bmod n)\\
p^{f_{l}}\leq x
}
}\frac{k\varphi(n)}{f_{l}p^{f_{l}\sigma/2}} & \ll k\sum_{\substack{1<l<n\\
(l,n)=1
}
}\varphi(n)\left(\frac{1}{n^{\sigma/2}l}+\frac{1}{n^{\sigma}}\right)\\
 & \ll k\varphi(n)\left(\frac{\ln n}{n^{\sigma/2}}+\frac{1}{n^{\sigma-1}}\right)\\
 & \ll k\left(n^{\frac{D\ln_{2}x}{2\ln x}}\ln n+n^{\frac{D\ln_{2}x}{\ln x}}\right)\\
 & \ll_{A}k\ln_{2}x.
\end{align*}
This proves that 
\[
Z_{2}<\exp(C_{2}(A)k\ln_{2}x)
\]
 for a positive constant $C_{2}(A).$

We finally consider $Z_{3}.$ For $Z_{3}$ we have

\begin{align*}
Z_{3} & <\prod_{\substack{\mathfrak{p}|n}
}\left(1-\frac{1}{N(\mathfrak{p})^{\sigma/2}}\right)^{-2k}\\
 & =\exp\left(O\left(k\sum_{\substack{\mathfrak{p}|n}
}\frac{1}{N(\mathfrak{p})^{\sigma/2}}\right)\right).
\end{align*}
Let $f_{p}$ denote the order of $p(\bmod\frac{n}{p^{v_{p}}}).$ It
follows from the inequalities $p^{f_{p}}>\frac{n}{p^{v_{p}}}$ and
$\sum_{p|n}1\ll\ln n\ll_{A}\ln_{2}x$ that

\begin{align*}
\sum_{\substack{\mathfrak{p}|n}
}\frac{1}{N(\mathfrak{p})^{\sigma/2}} & =\sum_{p|n}\frac{k\varphi(n)}{f_{p}\varphi(p^{v_{p}})}\frac{1}{p^{f_{p}\sigma/2}}\\
 & \ll\sum_{p|n}\left(\frac{p^{v_{p}}}{n}\right)^{\sigma/2}\frac{k\varphi(n)}{f_{p}\varphi(p^{v_{p}})}\\
 & \ll kn^{\frac{D\ln_{2}x}{2\ln x}}\sum_{p|n}1\ll_{A}k\ln_{2}x.
\end{align*}
This proves that 
\[
Z_{3}<\exp(C_{3}(A)k\ln_{2}x),
\]
 for a positive constant $C_{3}(A).$ This concludes the proof.
\end{proof}
We also need the following result.
\begin{lem}
\textup{\label{Lemma 2.3}(Hoeffding's inequality, \cite[Lemma 3.2]{A})}
Assume that $\{X_{k}\}_{k\geq1}$ is a sequence of independent random
variables, and $\ensuremath{\mathbb{P}(X_{k}=1)=\mathbb{P}(X_{k}=-1)=\frac{1}{2}}.$
Also, assume that $\{a_{k}\}_{k\geq1}$ is a sequence of real numbers
such that $\ensuremath{\sum_{k=1}^{\infty}a_{k}^{2}<\infty}.$ Then,
for any $\ensuremath{\lambda>0},$ we have

\[
\mathbb{P}\left(\sum_{k=1}^{\infty}a_{k}X_{k}\geq\lambda\right)\leq\exp\left(-\frac{\lambda^{2}}{2\sum_{k=1}^{\infty}a_{k}^{2}}\right).
\]
\end{lem}

\section{Proof of Theorem \ref{Theorem 1.2}}

In order to prove Theorem \ref{Theorem 1.2}, we first deduce two
auxiliary results.
\begin{prop}
\label{Proposition 5.1} Let $a$ and $m$ be two positive integers
with $\gcd(a,m)=1$ and, let $z$ be a real number and $\epsilon>0.$
Then, when $x$ is sufficiently large, there exists $\delta(\epsilon,z)>0$
such that the probability that the random variable 
\[
\sum_{\substack{p\equiv a(\bmod m)\\
p\leq x
}
}f(p)/p
\]
 falls within the real interval $(z-\epsilon,z+\epsilon)$ is greater
than $\delta(\epsilon,z)$.
\end{prop}
\begin{proof}
We arrange all primes satisfying $p\equiv a(\bmod m)$ in ascending
order as $\ensuremath{\{p_{1},p_{2},p_{3},\dots\}}.$ As the proof
of the case $z<0$ is similar to that of the case $z\geq0,$ we only
consider the case $z\geq0.$ Let us assume $\ensuremath{z\geq0}.$
We select $M_{0}\in\mathbb{Z}$ such that 
\[
\sum_{\substack{p\equiv a(\bmod m)\\
p\leq p_{M_{0}}
}
}\frac{1}{p}>z
\]
and
\[
\sum_{\substack{p\equiv a(\bmod m)\\
p\leq p_{M_{0}-1}
}
}\frac{1}{p}\leq z.
\]
Then we choose $M_{1}>M_{0}$ such that 
\[
\sum_{\substack{p\equiv a(\bmod m)\\
p\leq p_{M_{0}}
}
}\frac{1}{p}-\sum_{\substack{p\equiv a(\bmod m)\\
p_{M_{0}}<p\leq p_{M_{1}}
}
}\frac{1}{p}<z,
\]
 but 
\[
\sum_{\substack{p\equiv a(\bmod m)\\
p\leq p_{M_{0}}
}
}\frac{1}{p}-\sum_{\substack{p\equiv a(\bmod m)\\
p_{M_{0}}<p\leq p_{M_{1}-1}
}
}\frac{1}{p}\geq z.
\]
We, in a similar manner, continue to construct $M_{2},M_{3},\cdots.$
And we take 

\begin{align*}
f(p) & =1,p\leq p_{M_{0}}\:\textrm{or}\:p_{M_{2k-1}}<p\leq p_{M_{2k}},\;k=1,2,3,\cdots,p\equiv a(\bmod m)\\
f(p) & =-1,p_{M_{2k}}<p\leq p_{M_{2k+1}},\;k=0,1,2,\cdots,p\equiv a(\bmod m)
\end{align*}
By iteratively constructing $\ensuremath{M_{0},M_{1},\cdots,}$ it
holds that

\[
\left|\sum_{\substack{p\equiv a(\bmod m)\\
p\leq p_{M_{i}}
}
}\frac{f(p)}{p}-z\right|<\frac{1}{p_{M_{i}}}.
\]
Thus, we guarantee that these exists a sufficiently large $N_{0}$
such that
\begin{align*}
2\exp\left(-\frac{\epsilon^{2}}{4\sum_{N_{0}<p}\frac{1}{p^{2}}}\right) & <1-\delta_{2}\\
\left|\sum_{\substack{p\equiv a\,(\bmod m)\\
p\leq N_{0}
}
}\frac{f(p)}{p}-z\right| & <\frac{\epsilon}{2}.
\end{align*}

We, for sufficiently large $x$ (i.e., $x>N_{0}$), now decompose
the sum $\sum_{\substack{p\equiv a(\bmod m)\\
p<x
}
}\frac{f(p)}{p}$ into two parts:
\begin{align*}
\sum_{\substack{p\equiv a(\bmod m)\\
p<x
}
}\frac{f(p)}{p} & =\sum_{\substack{p\equiv a\,(\bmod m)\\
p\leq N_{0}
}
}\frac{f(p)}{p}+\sum_{\substack{p\equiv a\,(\bmod m)\\
N_{0}<p<x
}
}\frac{f(p)}{p}\\
 & =E_{1}+E_{2}
\end{align*}
It is easily seen that 
\[
\mathbb{P}\left(\left|E_{1}-z\right|<\epsilon/2\right)>\delta_{1}
\]
for some positive number $\delta_{1}.$ 

Since
\[
\ensuremath{\mathbb{P}\left(\sum_{k=1}^{\infty}a_{k}X_{k}\geq\lambda\right)=\mathbb{P}\left(\sum_{k=1}^{\infty}a_{k}X_{k}\leq-\lambda\right)},
\]
we, by Lemma \ref{Lemma 2.4}, obtain
\[
\mathbb{P}\left(\left|\sum_{k=1}^{\infty}a_{k}X_{k}\right|\geq\lambda\right)\leq2\exp\left(-\frac{\lambda^{2}}{2\sum_{k=1}^{\infty}a_{k}^{2}}\right).
\]
Then
\begin{align*}
\mathbb{P}\left(\left|E_{2}\right|<\epsilon/2\right) & =1-\mathbb{P}\left(\left|E_{2}\right|\geq\epsilon/2\right)\\
 & >1-2\exp\left(-\frac{\epsilon^{2}}{4\sum_{N_{0}<p}\frac{1}{p^{2}}}\right)>\delta_{2}.
\end{align*}
 Noting the independence of the random variables $E_{1}$ and $\ensuremath{E_{2}},$
we get
\begin{align*}
\mathbb{P}\left(\left|\sum_{\substack{p\equiv a(\bmod m)\\
p<x
}
}\frac{f(p)}{p}-z\right|<\epsilon\right) & >\mathbb{P}\left(\left(\left|E_{1}-z\right|<\epsilon/2\right)\wedge\left(\left|E_{2}\right|<\epsilon/2\right)\right)\\
 & \geq\mathbb{P}\left(\left|E_{1}-z\right|<\epsilon/2\right)\mathbb{P}\left(\left|E_{2}\right|<\epsilon/2\right)\\
 & >\delta_{1}\delta_{2}.
\end{align*}
 This completes the proof.
\end{proof}
\begin{prop}
\label{Proposition 5.2} Assume that $\chi$ is a character modulo
$m$. Let $C_{\chi}$ be a constant depending on $\chi$ with $C_{\bar{\chi}}=\overline{C_{\chi}}.$
Assume $S$ is a non-empty subset of the reduced residue system $\left(\ensuremath{\mathbb{Z}/m\mathbb{Z}}\right)^{\times},$
and $\{\gamma_{a}\}_{a\in\left(\ensuremath{\mathbb{Z}/m\mathbb{Z}}\right)^{\times}}$
is a set of real variables. Then the multivariate function 
\[
F(\gamma_{1},\cdots,\gamma_{\varphi(m)}):=\frac{1}{\varphi(m)}\sum_{b\in S}\sum_{\chi\in\widehat{(\mathbb{Z}/m\mathbb{Z})^{\times}}}\overline{\chi}(b)\exp\left(\sum_{a\in\left(\ensuremath{\mathbb{Z}/m\mathbb{Z}}\right)^{\times}}\chi(a)\gamma_{a}+C_{\chi}\right)
\]
takes non-negative values if and only if the coefficients of the real
characters in the expansion of the characteristic function of $S$
are all non-negative, and the coefficients of the complex characters
are all zero. Otherwise, this function can take arbitrarily large
negative values.
\end{prop}
\begin{proof}
We rewrite the summation in the expression of $F(\gamma_{1},\cdots,\gamma_{\varphi(m)})$
as two sums:

\begin{align*}
F(\gamma_{1},\cdots,\gamma_{\varphi(m)}) & =\sum_{\chi\,\text{real}}b_{\chi}\exp\left(\sum_{a\in\left(\ensuremath{\mathbb{Z}/m\mathbb{Z}}\right)^{\times}}\chi(a)\gamma_{a}\right)\\
 & +\sum_{\overline{\chi}\,\text{complex}}2|b_{\chi}|\exp\left(\sum_{a\in\left(\ensuremath{\mathbb{Z}/m\mathbb{Z}}\right)^{\times}}\chi^{(1)}(a)\gamma_{a}\right)\cos\left(\sum_{a\in\left(\ensuremath{\mathbb{Z}/m\mathbb{Z}}\right)^{\times}}\chi^{(2)}(a)\gamma_{a}+\theta_{\chi}\right),
\end{align*}
where
\begin{align*}
b_{\chi} & =\sum_{b\in S}\overline{\chi}(b)\exp\left(C_{\chi}\right),\theta_{\chi}=\text{arg}(b_{\chi})\\
\chi^{(1)}(a) & =\text{Re}(\chi(a)),\quad\chi^{(2)}(a)=\text{Im}(\chi(a)).
\end{align*}

Define three linear transformations:

\begin{align}
\begin{alignedat}{1}t_{\chi} & =\sum_{a\in\left(\ensuremath{\mathbb{Z}/m\mathbb{Z}}\right)^{\times}}\chi(a)\gamma_{a},\quad\chi\,\text{real},\\
r_{\chi}^{(1)} & =\sum_{a\in\left(\ensuremath{\mathbb{Z}/m\mathbb{Z}}\right)^{\times}}\chi^{(1)}(a)\gamma_{a},\quad\chi\,\text{complex},\\
r_{\chi}^{(2)} & =\sum_{a\in\left(\ensuremath{\mathbb{Z}/m\mathbb{Z}}\right)^{\times}}\chi^{(2)}(a)\gamma_{a},\quad\chi\,\text{complex},
\end{alignedat}
\label{eq:4-1}
\end{align}
 we will prove that this is a full-rank linear transformation.

Suppose $\chi_{1},\cdots,\chi_{r_{1}}$ are real characters, and $\chi_{r_{1}+1},\cdots,\chi_{r_{1}+2r_{2}}$
are complex characters, satisfying $\chi_{r_{1}+j}=\overline{\chi}_{r_{1}+r_{2}+j}$
for $\ensuremath{j=1,\cdots,r_{2}}.$ Then, the determinant of the
matrix associated with the given linear transformation is proportional
to the determinant of the matrix 

\begin{equation}
\begin{pmatrix}\chi_{1}(1) & \cdots & \chi_{1}(m-1)\\
\vdots & \ddots & \vdots\\
\chi_{r_{1}+2r_{2}}(1) & \cdots & \chi_{r_{1}+2r_{2}}(m-1)
\end{pmatrix},\label{eq:4-3}
\end{equation}
differing only by a nonzero scalar factor. By the orthogonality relations
of characters, it follows that the determinant of the matrix in \eqref{eq:4-3}
is nonzero. Consequently, the linear transformation in \eqref{eq:4-1}
is full rank.

Thus, through this linear transformation, we obtain

\[
F(\gamma_{1},\cdots,\gamma_{m-1})=G(t_{\chi_{1}},\cdots,r_{\chi_{r_{1}+1}}^{(1)},\cdots,r_{\chi_{r_{1}+r_{2}}}^{(2)}),
\]
where 
\[
G(t_{\chi_{1}},\cdots,r_{\chi_{r_{1}+1}}^{(1)},\cdots,r_{\chi_{r_{1}+r_{2}}}^{(2)})=\sum_{\chi\,\text{real}}b_{\chi}\exp(t_{\chi})+\sum_{\overline{\chi}\,\text{complex}}2|b_{\chi}|\exp(r_{\chi}^{(1)})\cos\left(r_{\chi}^{(2)}+\theta_{\chi}\right).
\]
Thus, the function $G\geq0$ if and only if all $|b_{\chi}|=0$ for
all complex characters $\ensuremath{\chi}.$ This is equivalent to
the condition that for all complex characters $\ensuremath{\chi},$
we have $\ensuremath{\sum_{b\in S}\overline{\chi}(b)=0},$ and for
all real characters $\chi$, we have $\ensuremath{b_{\chi}>0}.$

Note that the expansion of the characteristic function of $S$ is

\[
1_{S}(x)=\frac{1}{\varphi(m)}\sum_{\chi\in\widehat{(\mathbb{Z}/m\mathbb{Z})^{\times}}}\left(\sum_{a\in S}\overline{\chi}(a)\right)\chi(x).
\]
Therefore, $\ensuremath{F\geq0}$ is equivalent to that real character
coefficients in the characteristic function of $\ensuremath{S}$ are
all nonnagetive and all complex character coefficients are vanishing.
\end{proof}
We are now ready to prove Theorem \ref{Theorem 1.2}.

\noindent{\it Proof of Theorem \ref{Theorem 1.2}.} If the condition
that the coefficients of the real characters in the expansion of the
characteristic function of $S$ according to the characters of $(\mathbb{Z}/m\mathbb{Z})^{\times}$
are all non-negative, and the coefficients of the complex characters
are all zero is not satisfied, then, using characteristic function
decomposition:

\[
1_{A}(x)=\frac{1}{\varphi(m)}\sum_{a\in S}\sum_{\chi\in\widehat{\left(\ensuremath{\mathbb{Z}/m\mathbb{Z}}\right)^{\times}}}\overline{\chi}(a)\chi(x),
\]
 we obtain

\begin{align*}
\sum_{n\leq x}\frac{1_{A}(n)f(n)}{n} & =\frac{1}{\varphi(m)}\sum_{n\leq x}\sum_{a\in S}\sum_{\chi\in\widehat{\left(\ensuremath{\mathbb{Z}/m\mathbb{Z}}\right)^{\times}}}\frac{\overline{\chi}(a)\chi(n)f(n)}{n}\\
 & =\frac{1}{\varphi(m)}\sum_{a\in S}\sum_{\chi\in\widehat{\left(\ensuremath{\mathbb{Z}/m\mathbb{Z}}\right)^{\times}}}\sum_{n\leq x}\frac{\overline{\chi}(a)\chi(n)f(n)}{n}\\
 & =\frac{1}{\varphi(m)}\sum_{a\in S}\sum_{\chi\in\widehat{\left(\ensuremath{\mathbb{Z}/m\mathbb{Z}}\right)^{\times}}}\overline{\chi}(a)\left(\prod_{p\leq x}\left(1-\frac{\chi(p)f(p)}{p}\right)^{-1}-\sum_{\substack{n>x\\
P(n)<x
}
}\frac{\chi(n)f(n)}{n}\right)\\
 & =:U_{1}-U_{2}.
\end{align*}

We now prove that for sufficiently large $\ensuremath{x},$ there
exists a positive constant $\delta$ independent of $x$ such that

\[
P(U_{1}<-1)>\delta
\]
and 
\[
P\left(|U_{2}|<\frac{1}{\ln x}\right)>1-O\left(\exp\left(-\exp\left(\frac{\ln x}{C\ln^{2}x}\right)\right)\right),
\]
where $C$ is a positive constant.

In fact, after basic algebraic manipulation, we get:

\begin{align*}
U_{1} & =\frac{1}{\varphi(m)}\sum_{b\in S}\sum_{\chi\in\widehat{\left(\ensuremath{\mathbb{Z}/m\mathbb{Z}}\right)^{\times}}}\overline{\chi}(b)\prod_{p\leq x}\left(1-\frac{\chi(p)f(p)}{p}\right)^{-1}\\
 & =\frac{1}{\varphi(m)}\sum_{b\in S}\sum_{\chi\in\widehat{\left(\ensuremath{\mathbb{Z}/m\mathbb{Z}}\right)^{\times}}}\overline{\chi}(b)\exp\left(z_{\chi}(x)+C_{\chi}+e_{\chi}(x)\right)
\end{align*}
where

\begin{align*}
z_{\chi}(x) & :=\sum_{a\in\left(\ensuremath{\mathbb{Z}/m\mathbb{Z}}\right)^{\times}}\chi(a)\gamma_{a}(x),\\
\gamma_{a}(x) & :=\sum_{\substack{p\equiv a\bmod m\\
p\leq x
}
}\frac{f(p)}{p},\\
C_{\chi} & :=\sum_{p}\ln\left(1-\frac{\chi(p)f(p)}{p}\right)^{-1}-\frac{\chi(p)f(p)}{p},\\
e_{\chi}(x) & :=\sum_{p\geq x}\ln\left(1-\frac{\chi(p)f(p)}{p}\right)^{-1}-\frac{\chi(p)f(p)}{p}.
\end{align*}

Define

\[
U_{1}^{\prime}:=\frac{1}{\varphi(m)}\sum_{b\in S}\sum_{\chi\in\widehat{\left(\ensuremath{\mathbb{Z}/m\mathbb{Z}}\right)^{\times}}}\overline{\chi}(b)\exp\left(z_{\chi}(x)+C_{\chi}\right).
\]
By Proposition \ref{Proposition 5.2}, there exists an interval $I_{a}=(z_{a}-\epsilon_{a},z_{a}+\epsilon_{a})$
independent of $x$ such that if $\ensuremath{\gamma_{a}(x)\in I_{a}},$
then $\ensuremath{U_{1}^{\prime}<-2}.$ Noting that the random variables
$\gamma_{a}(x)$ are independent, we have the probabilistic inequality:

\[
\mathbb{P}(U_{1}^{\prime}<-2)>\prod_{a\in\left(\ensuremath{\mathbb{Z}/m\mathbb{Z}}\right)^{\times}}\mathbb{P}(\gamma_{a}(x)\in I_{a}).
\]
By Proposition \ref{Proposition 5.1}, there exists a positive constant
$\delta_{a}$ independent of $x$ such that

\[
\mathbb{P}(\gamma_{a}(x)\in I_{a})>\delta_{a}.
\]
Thus

\[
\mathbb{P}(U_{1}^{\prime}<-2)>\prod_{a\in\left(\ensuremath{\mathbb{Z}/m\mathbb{Z}}\right)^{\times}}\delta_{a}.
\]
Since $e_{\chi}(x)$ becomes sufficiently small as $x$ grows large,
we obtain that, for sufficiently large $x,$

\[
\mathbb{P}(U_{1}<-1)>\mathbb{P}(U_{1}^{\prime}<-2)>\prod_{a\in\left(\ensuremath{\mathbb{Z}/m\mathbb{Z}}\right)^{\times}}\delta_{a}.
\]

On the other hand, note that

\[
|U_{2}|<\frac{1}{\varphi(m)}\sum_{a\in S}\sum_{\chi\in\widehat{\left(\ensuremath{\mathbb{Z}/m\mathbb{Z}}\right)^{\times}}}U'_{2,\chi}
\]
with

\[
U'_{2,\chi}=\left|\sum_{\substack{n>x\\
P(n)<x
}
}\frac{\chi(n)f(n)}{n}\right|.
\]

Let $\ensuremath{k=x^{\frac{1}{D\ln_{2}x}},\sigma=2-\frac{D\ln_{2}x}{\ln x},R=2k^{2/\sigma}}.$
Then
\begin{align*}
\mathbb{E}\left(\left(U_{2,\chi}^{\prime}\right)^{2k}\right) & =\mathbb{E}\left(\sum_{\substack{n_{i}>x\\
P(n_{i})<x
}
}\frac{f(\prod_{i=1}^{2k}n_{i})\prod_{i=1}^{k}\chi(n_{i})\overline{\chi(n_{k+i})}}{\prod_{i=1}^{2k}n_{i}}\right)\\
 & =\sum_{\substack{n_{i}>x\\
P(n_{i})<x\\
\prod_{i=1}^{2k}n_{i}\in\mathbb{Z}^{2}
}
}\frac{f(\prod_{i=1}^{2k}n_{i})\prod_{i=1}^{k}\chi(n_{i})\overline{\chi(n_{k+i})}}{\prod_{i=1}^{2k}n_{i}}\\
 & <\sum_{\substack{n_{i}>x\\
P(n_{i})<x\\
\prod_{i=1}^{2k}n_{i}\in\mathbb{Z}^{2}
}
}\frac{1}{\prod_{i=1}^{2k}n_{i}}<\sum_{\substack{m>x^{k}\\
P(m)<x
}
}\frac{d_{2k}(m^{2})}{m^{2}}.
\end{align*}

By using Lemma \ref{Lemma 2.5}, we get

\[
\mathbb{E}\left(\left(U_{2,\chi}^{\prime}\right)^{2k}\right)<\sum_{\substack{m>x^{k}\\
P(m)<x
}
}\frac{d_{2k}(m^{2})}{m^{2}}<\frac{1}{x^{k(2-\sigma)}}\exp\left(kO(\ln_{2}x)\right).
\]
Applying Markov's inequality, we deduce that, for sufficiently large
$D$ and $\ensuremath{x},$

\begin{align*}
\mathbb{P}(U_{2,\chi}^{\prime} & >\frac{1}{\varphi(m)\ln x})=\mathbb{P}(\left(U_{2,\chi}^{\prime}\right)^{2k}>\left(\frac{1}{\varphi(m)\ln x}\right)^{2k})<\frac{\mathbb{E}\left(\left(U_{2,\chi}^{\prime}\right)^{2k}\right)}{\left(\frac{1}{\varphi(m)\ln x}\right)^{2k}}\\
 & <\exp\left(-\frac{D}{5}\ln_{2}x\exp\left(\frac{\ln x}{D\ln_{2}x}\right)\right).
\end{align*}
Therefore,

\[
\mathbb{P}\left(|U_{2}|<\frac{1}{\ln x}\right)>\mathbb{P}\left(\bigcap_{\chi}U_{2,\chi}^{\prime}<\frac{1}{\varphi(m)\ln x}\right)>1-O\left(\exp\left(-\exp\left(\frac{\ln x}{C\ln_{2}x}\right)\right)\right),
\]
 where $C$ is a positive constant. This implies that, for sufficiently
large $\ensuremath{x},$

\[
\mathbb{P}\left(\sum_{n\leq x}\frac{1_{A}(n)f(n)}{n}<0\right)>\mathbb{P}(U_{1}<-1)+\mathbb{P}\left(|U_{2}|<\frac{1}{\ln x}\right)-1>\frac{1}{2}\prod_{a\in(\mathbb{Z}/m\mathbb{Z})^{\times}}\delta_{a}.
\]

If the condition is satisfied, then, using the characteristic function
decomposition:

\[
1_{A}(x)=\sum_{\chi\in\widehat{\left(\ensuremath{\mathbb{Z}/m\mathbb{Z}}\right)^{\times}}}C_{\chi}\chi(x),
\]
where $C_{\chi}\geq0$ and not all of the $C_{\chi}$ are 0, we arrive
at

\[
\sum_{n\leq x}\frac{1_{A}(n)f(n)}{n}=\sum_{\chi\in\widehat{\left(\ensuremath{\mathbb{Z}/m\mathbb{Z}}\right)^{\times}}}C_{\chi}\sum_{n\leq x}\frac{\chi(n)f(n)}{n}
\]
Spliting the sum $\sum_{n\leq x}\frac{\chi(n)f(n)}{n}$ , we get

\begin{align*}
\sum_{n\leq x}\frac{\chi(n)f(n)}{n} & =\prod_{\substack{p\leq x\\
(p,m)=1
}
}\left(1-\frac{f_{\chi}(p)}{p}\right)^{-1}-\sum_{\substack{P(n)\leq x\\
n>x,(n,m)=1
}
}\frac{\chi(n)f(n)}{n}\\
 & =:F_{1}-F_{2}
\end{align*}

Applying the same trick as in the proofs of Propositions 3.1 and 3.2
we can deduce that

\[
\mathbb{P}(F_{1}<\delta)<\textrm{\ensuremath{\delta^{k}}}\exp(kC\ln_{2}x),
\]
and
\[
\mathbb{P}(F_{2}>\delta)<\textrm{\ensuremath{\delta^{k}}}\exp(kC\ln_{2}x),
\]
where $C>0,\delta=\left(\frac{\ln_{2}x}{\ln x}\right)^{2C},k=x^{\frac{1}{\ln_{2}x}}$.
By the basic probability inequality:

\[
\mathbb{P}(F_{1}-F_{2}\geq0)\geq\mathbb{P}\left((F_{1}\geq\delta)\land(F_{2}\leq\delta)\right)\geq\mathbb{P}(F_{1}\geq\delta)+\mathbb{P}(F_{2}\leq\delta)-1
\]
we have

\begin{align*}
\mathbb{P}\left(\sum_{n\leq x}\frac{\chi(n)f(n)}{n}\geq0\right) & =1-\exp\left(-O\left(\exp\left(\frac{\ln x}{C\ln_{2}x}\right)\right)\right).
\end{align*}
This implies that

\begin{align*}
\mathbb{P}\left(\sum_{\chi\in\widehat{\left(\ensuremath{\mathbb{Z}/m\mathbb{Z}}\right)^{\times}}}C_{\chi}\sum_{n\leq x}\frac{\chi(n)f(n)}{n}\geq0\right) & \geq\prod_{\chi\in\widehat{\left(\ensuremath{\mathbb{Z}/m\mathbb{Z}}\right)^{\times}}}\mathbb{P}\left(\sum_{n\leq x}\frac{\chi(n)f(n)}{n}\geq0\right)\\
 & \geq\prod_{\chi\in\widehat{\left(\ensuremath{\mathbb{Z}/m\mathbb{Z}}\right)^{\times}}}\left[1-\exp\left(-O\left(\exp\left(\frac{\ln x}{C\ln_{2}x}\right)\right)\right)\right]\\
 & =1-\exp\left(-O\left(\exp\left(\frac{\ln x}{C\ln_{2}x}\right)\right)\right).
\end{align*}
This concludes the proof of Theorem \ref{Theorem 1.2}. \qed

\section{Proof of Theorem \ref{Theorem 1.1}}

To prove Theorem \ref{Theorem 1.1}, we first show two auxiliary results.
\begin{prop}
\label{p3.1} Let $\delta$ and $k$ be two positive constants. For
$K_{n}=Q(\zeta_{n}),$ we defined that

\[
Y_{x,K_{n}}=\prod_{N(\mathfrak{p})\leq x}\left(1-\frac{f(\mathfrak{p})}{N(\mathfrak{p})}\right)^{-1}
\]
Then for $n<(\ln x)^{A},k\leq x^{\frac{1}{\ln_{2}x}}$ with $A>0,$
we have

\[
\mathbb{P}(Y_{x,K_{n}}<\delta)=\textrm{\ensuremath{\delta^{k}}}\exp(kC(A)\ln_{2}x),
\]
where $C(A)$ is a positive constant depending on $A.$
\end{prop}
\begin{proof}
Notice that

\begin{align*}
\mathbb{E}(Y_{x,K_{n}}^{-k}) & =\prod_{N(\mathfrak{p})\leq x}\frac{(1-\frac{1}{N(\mathfrak{p})})^{k}+(1+\frac{1}{N(\mathfrak{p})})^{k}}{2}<Z_{1}Z_{2}Z_{3}.
\end{align*}
Using Lemma \ref{Lemma 2.5}, we obtain

\[
\mathbb{E}(Y_{x,K_{n}}^{-k})<\exp(kO_{A}(\ln_{2}x)).
\]
Applying Markov's inequality we get that

\[
\mathbb{P}(Y_{x,K_{n}}<\delta)=\mathbb{P}(Y_{x,K_{n}}^{-k}>\delta^{-k})<\delta^{k}\mathbb{E}(Y_{x,K_{n}}^{-k})<\delta^{k}\exp(kO_{A}(\ln_{2}x)).
\]
This completes the proof.
\end{proof}
\begin{cor}
\label{C4.1} For $\delta=\left(\frac{\ln_{2}x}{\ln x}\right)^{2C(A)},$
we have that

\[
\mathbb{P}(Y_{x,K_{n}}<\delta)<\exp\left(-\frac{C(A)}{2}\ln_{2}x\exp\left(\frac{\ln x}{\ln_{2}x}\right)\right).
\]
\end{cor}
\begin{proof}
Taking $k=x^{\frac{1}{\ln_{2}x}}$ in Proposition \ref{p3.1}, we
derive that for large enough $x,$
\begin{align*}
\mathbb{P}(Y_{x,K_{n}}<\delta) & =\textrm{\ensuremath{\delta^{k}}}\exp(kC(A)\ln_{2}x)\\
 & =\exp\left(-C(A)\exp\left(\frac{\ln x}{\ln_{2}x}\right)\ln_{2}x+2C(A)\exp\left(\frac{\ln x}{\ln_{2}x}\right)\ln_{3}x\right)\\
 & <\exp\left(-\frac{C(A)}{2}\ln_{2}x\exp\left(\frac{\ln x}{\ln_{2}x}\right)\right).
\end{align*}
As desired.
\end{proof}
\begin{prop}
\label{Proposition 3.2} Let $\delta$ be a positive number. For $K_{n}=Q(\zeta_{n}),$
defined that

\[
Z_{x,K_{n}}:=\sum_{\substack{N(\mathfrak{a})>x\\
P(\mathfrak{a})\leq x
}
}\frac{f(\mathfrak{a})}{N(\mathfrak{a})}.
\]
Then, for $n<(\ln x)^{A}$ with $A>0,$ we have

\[
\mathbb{P}(Z_{x,K_{n}}>\delta)<\frac{\textrm{Exp}(-\frac{kD}{2}\ln_{2}x)}{\delta^{2k}}.
\]
where $k=x^{\frac{1}{D\ln_{2}x}}$ with large enough $D>0$ depending
on $A$.
\end{prop}
\begin{proof}
Let 
\[
\sigma=2-\frac{D\ln_{2}x}{\ln x}.
\]
 We assume that $x$ is sufficiently large so that $\sigma>1.99.$

Let us first estimate the expectation of the random variable $Z_{x,K_{n}}^{2k}.$
Note that

\begin{align*}
\mathbb{E}(Z_{x,K_{n}}^{2k}) & =\sum_{\substack{N(\mathfrak{a_{\mathit{\textrm{i}}}})>x\\
\mathfrak{a_{\mathit{\textrm{1}}}\cdots\mathfrak{a_{\mathit{2k}}}=\mathfrak{b^{\mathit{2}}}}\\
P(\mathfrak{a_{\mathit{\textrm{i}}}})\leq x
}
}\frac{1}{N(\mathfrak{b})^{2}}<\sum_{\substack{N(\mathfrak{b})>x^{k}\\
P(\mathfrak{b})\leq x
}
}\frac{d_{2k}(\mathfrak{b}^{2})}{N(\mathfrak{b})^{2}}\\
 & <\frac{1}{x^{k(2-\sigma)}}\sum_{\substack{P(\mathfrak{b})\leq x}
}\frac{d_{2k}(\mathfrak{b})}{N(\mathfrak{b})^{\sigma}}\\
 & =\frac{1}{x^{k(2-\sigma)}}\prod_{N\mathfrak{(\mathfrak{p})}\leq x}\left(\frac{(1-\frac{1}{N(\mathfrak{p})^{\sigma/2}})^{-2k}+(1+\frac{1}{N(\mathfrak{p})^{\sigma/2}})^{-2k}}{2}\right),
\end{align*}
where $P(\mathfrak{a})$ denotes the maximal norms of the prime ideals
in all of the prime ideal decompositions of $\mathfrak{a}$ and 
\[
d_{k}(\mathfrak{a}):=\sum_{\mathfrak{a}=\mathfrak{b}_{1}\mathfrak{b}_{2}\cdots\mathfrak{b}_{k}}1.
\]
Using Lemma \ref{Lemma 2.1}, we obtain

\[
\prod_{N\mathfrak{(\mathfrak{p})}\leq x}\left(\frac{(1-\frac{1}{N(\mathfrak{p})^{\sigma/2}})^{-2k}+(1+\frac{1}{N(\mathfrak{p})^{\sigma/2}})^{-2k}}{2}\right)=Z_{1}Z_{2}Z_{3}
\]

Then we apply Lemma \ref{Lemma 2.5} to get that

\[
\mathbb{E}(Z_{x,K_{n}}^{2k})<\exp\left(-kD\ln_{2}x+(C_{1}(A)+C_{2}(A)+C_{3}(A))k\ln_{2}x\right)
\]
Choose $D$ large enough such that $C_{1}(A)+C_{2}(A)+C_{3}(A)<D/2.$
Then

\[
\mathbb{E}(Z_{x,K_{n}}^{2k})<\exp\left(-\frac{kD}{2}\ln_{2}x\right)
\]
By Markov's inequality, we get

\begin{align*}
\mathbb{P}(Z_{x,K_{n}} & >\delta)=P(Y_{x,2}^{2k}>\delta^{2k})<\frac{\mathbb{E}(Z_{x,K_{n}}^{2k})}{\delta^{2k}}<\frac{\textrm{Exp}(-\frac{kD}{2}\ln_{2}x)}{\delta^{2k}}.
\end{align*}
This concludes the proof.
\end{proof}
\begin{cor}
\label{c4.2} For $\delta=\left(\frac{\ln_{2}x}{\ln x}\right)^{2C(A)},$
we have

\[
\mathbb{P}(Z_{x,K_{n}}>\delta)<\exp\left(-C(A)\exp\left(\frac{\ln x}{10C(A)\ln_{2}x}\right)\ln_{2}x\right).
\]
\end{cor}
\begin{proof}
For $\delta=\left(\frac{\ln_{2}x}{\ln x}\right)^{2C(A)},$ we, by
Proposition \ref{Proposition 3.2}, get

\[
\mathbb{P}(Z_{x,K_{n}}<\delta)<\frac{\exp(-\frac{kD}{2}\ln_{2}x)}{\delta^{k}}<\exp\left(-\frac{kD}{2}\ln_{2}x-4kC(A)\ln_{3}x+4kC(A)\ln_{2}x\right).
\]
If we take $D=10C(A),$ then for $x$ large enough we obtain

\[
\mathbb{P}(Z_{x,K_{n}}<\delta)<\exp\left(-C(A)\ln_{2}x\exp\left(\frac{\ln x}{10C(A)\ln_{2}x}\right)\right).
\]
As desired.
\end{proof}
We are now in a position to prove Theorem \ref{Theorem 1.1}.

\noindent{\it Proof of Theorem \ref{Theorem 1.1}.} It is easy to
see that 
\[
S_{x,K_{n}}=Y_{x,K_{n}}-Z_{x,K_{n}}.
\]
 If we take 
\[
\delta=\left(\frac{\ln_{2}x}{\ln x}\right)^{2C(A)},
\]
 then, by Corollaries \ref{C4.1} and \ref{c4.2},

\[
\mathbb{P}(Y_{x,K_{n}}<\delta)<\exp\left(-\frac{C(A)}{2}\ln_{2}x\exp\left(\frac{\ln x}{\ln_{2}x}\right)\right)
\]
and 

\[
\mathbb{P}(Z_{x,K_{n}}>\delta)<\exp\left(-C(A)\ln_{2}x\exp\left(\frac{\ln x}{10C(A)\ln_{2}x}\right)\right).
\]

Using the basic probability inequality

\[
\mathbb{P}(S_{x,K_{n}}\geq0)\geq\mathbb{P}\left((Y_{x,K_{n}}\geq\delta)\land(Z_{x,K_{n}}\leq\delta)\right)\geq\mathbb{P}(Y_{x,K_{n}}\geq\delta)+\mathbb{P}(Z_{x,K_{n}}\leq\delta)-1,
\]
we obtain that

\[
\mathbb{P}(S_{x,K_{n}}\geq0)>1-\exp\left(-C_{5}(A)\ln_{2}x\exp\left(\frac{\ln x}{C_{4}(A)\ln_{2}x}\right)\right).
\]
This completes the proof of Theorem \ref{Theorem 1.1}. \qed

\section{Proof of Theorem \ref{Theorem 1.3}}

In order to show Theorem \ref{Theorem 1.3}, we first derive the following
two results.
\begin{prop}
\label{Proposition 4.1}Suppose $\varrho(n)$ satisfies the same conditions
as in Theorem \ref{Theorem 1.3}. Define

\[
I_{x,1}:=\prod_{p\leq x}\left(1+\frac{\varrho(p)}{p^{\frac{m+2}{2}}}+\frac{\varrho(p^{2})}{p^{2(\frac{m+2}{2})}}+\cdots\right).
\]
Then, for $\delta>0$ and $k>0,$ we have

\[
\mathbb{P}(I_{x,1}<\delta)<\delta^{k}\exp(O(k\ln_{2}x)).
\]
\end{prop}
\begin{proof}
It is easy to see that

\[
I_{x,1}=\prod_{p\leq x}\left(\frac{1}{1-\frac{2\cos\theta_{p}}{p}+\frac{1}{p^{2}}}\right).
\]
Thus, by applying inequality 
\[
\left|1-\frac{2\cos\theta_{p}}{p}+\frac{1}{p^{2}}\right|<(1+\frac{1}{p})^{2},
\]
 we get

\begin{align*}
\mathbb{E}(I_{x,1}^{-k}) & =\prod_{p\leq x}\frac{2}{\pi}\int_{0}^{\pi}\left(1-\frac{2\cos\theta_{p}}{p}+\frac{1}{p^{2}}\right)^{k}\sin^{2}\theta d\theta\\
 & <\prod_{p\leq x}\left(1+\frac{1}{p}\right)^{2k}\\
 & =\exp(O(k\ln_{2}x)).
\end{align*}
Applying Markov's inequality, we obtain

\[
\mathbb{P}(I_{x,1}<\delta)=P(I_{x,1}^{-k}>\delta^{-k})<\frac{\mathbb{E}(I_{x,1}^{-k})}{\delta^{-k}}<\delta^{k}\exp(O(k\ln_{2}x)).
\]
This finishes the proof.
\end{proof}
\begin{prop}
\label{Proposition 4.2} Suppose $\varrho(n)$ satisfies the same
conditions as in Theorem \ref{Theorem 1.3}. Define

\[
I_{x,2}:=\sum_{\substack{n>x\\
P(n)\leq x
}
}\frac{\varrho(n)}{n^{\frac{m+2}{2}}},
\]
where $P(n)$ denotes the greatest prime divisor of $n.$ Then, for
any $\delta>0$ and $k=x^{\frac{1}{D\ln_{2}x}},$ we have

\[
\mathbb{P}(I_{x,2}<\delta)<\frac{\exp(-\frac{D}{2}k\ln_{2}x)}{\delta^{2k}},
\]
where $D>0$ is a large constant.
\end{prop}
\begin{proof}
We first give an upper bound for the expectation of the random variable
$I_{x,2}^{2k}.$

When $k$ is odd, $\varrho(p^{k})$ is an odd polynomial in $\varrho(p),$
and when $k$ is even, $\varrho(p^{k})$ is an even polynomial in
$\varrho(p).$ This means that $\mathbb{E}\left(\varrho(n_{1})\cdots\varrho(n_{2k})\right)=0,$
unless the product $n_{1}\cdots n_{2k}$ is a square. Thus

\begin{align*}
\mathbb{E}(I_{x,2}^{2k}) & =\sum_{\substack{n_{1}\cdot n_{2}\cdot\cdots\cdot n_{2k}\in\mathbb{Z}^{2}\\
n_{i}>x\\
P(n_{i})\leq x
}
}\frac{\mathbb{E}\left(\varrho(n_{1})\cdots\varrho(n_{2k})\right)}{(n_{1}\cdots n_{2k})^{\frac{m+2}{2}}}.
\end{align*}

Now we suppose that $n_{1}\cdot n_{2}\cdot\cdots\cdot n_{2k}=u^{2},$
then $u>x^{k}.$ For $i=1,\cdots,2k,p\leq x,$ we define

\[
v_{p,i}:=\textrm{ord}_{p}(n_{i}),\quad u_{p}:=\textrm{ord}_{p}(u),
\]
where $\textrm{ord}_{p}$ is the $p$-adic valuation. Then 
\[
v_{p,1}+\cdots+v_{p,2k}=2u_{p}.
\]
For the expectation, we have

\begin{align*}
\mathbb{E}\left(\prod_{i=1}^{2k}\varrho(n_{i})\right) & =\prod_{p}\mathbb{E}\left(\prod_{i=1}^{2k}\varrho(p^{v_{p,i}})\right)\\
 & =\prod_{p}\left[p^{mu_{p}}\times\frac{2}{\pi}\int_{0}^{\pi}\left(\prod_{i=1}^{2k}\frac{\sin(v_{p,i}+1)\theta_{p}}{\sin\theta_{p}}\right)\sin^{2}\theta_{p}d\theta_{p}\right]\\
 & =u^{m}\prod_{p}\left[\frac{2}{\pi}\int_{0}^{\pi}\left(\prod_{i=1}^{2k}\frac{\sin(v_{p,i}+1)\theta_{p}}{\sin\theta_{p}}\right)\sin^{2}\theta_{p}d\theta_{p}\right].
\end{align*}

Using the elementary inequality: 
\[
\left|\frac{\sin(v_{p,i}+1)\theta_{p}}{\sin\theta_{p}}\right|\leq v_{p,i}+1
\]
 for $i=1,\cdots,2k$ and $p\leq x,$ we obtain

\[
\left|\int_{0}^{\pi}\left(\prod_{i=1}^{2k}\frac{\sin(v_{p,i}+1)\theta_{p}}{\sin\theta_{p}}\right)\sin^{2}\theta_{p}d\theta_{p}\right|<\prod_{i=1}^{2k}(v_{p,i}+1)\leq2^{u_{p}}.
\]
From this we deduce that

\[
\mathbb{E}\left(\prod_{i=1}^{2k}\varrho(n_{i})\right)<u^{m}\prod_{p}2^{u_{p}}=u^{m}2^{\Omega(u)},
\]
where $\Omega(u)$ represents the number of prime factors of the integer
$u$ (counting multiplicities). Then

\begin{equation}
\mathbb{E}(I_{x,2}^{2k})<\sum_{\substack{n_{1}\cdot n_{2}\cdot\cdots\cdot n_{2k}=u^{2}\\
n_{i}>x\\
P(n_{i})\leq x
}
}\frac{u^{m}2^{\Omega(u)}}{u^{m+2}}<\sum_{\substack{u>x^{k}\\
P^{+}(u)\leq x
}
}\frac{2^{\Omega(u)}}{u^{2}}\label{eq:4-2}
\end{equation}

From Lemma \ref{Lemma 2.4}, we get 

\[
\mathbb{E}(I_{x,2}^{2k})<\frac{1}{x^{k(2-\sigma)}}\exp(kO(\ln_{2}x)).
\]
Applying Markov's inequality, we get

\[
P(I_{x,2}<\delta)=P(I_{x,1}^{2k}<\delta^{2k})<\frac{\exp(O(k\ln_{2}x))}{x^{k(2-\sigma)}\delta^{2k}}.
\]
From this, we can easily obtain the desired upper bound for the probability.
\end{proof}
We are now ready to show Theorem \ref{Theorem 1.3}.

\noindent{\it Proof of Theorem \ref{Theorem 1.3}.} It is easy to
see that $I_{x}^{(m)}=I_{x,1}-I_{x,2}.$ Take $k=\exp(\frac{\ln x}{D\ln_{2}x}),\delta=\left(\frac{\ln_{2}x}{\ln x}\right)^{2C},$
where $C$ and $D$ are large enough. Then, by applying Propositions
\ref{Proposition 4.1} and \ref{Proposition 4.2}, we have

\[
\mathbb{P}(I_{x,1}<\delta)<\exp\left(-\frac{C}{2}\ln_{2}x\exp\left(\frac{\ln x}{D\ln_{2}x}\right)\right)
\]
and 

\[
\mathbb{P}(I_{x,2}>\delta)<\exp\left(-\frac{D}{5}\ln_{2}x\exp\left(\frac{\ln x}{D\ln_{2}x}\right)\right).
\]

Using the basic probability inequality:

\[
\mathbb{P}(I_{x}^{(m)}\geq0)\geq\mathbb{P}\left((I_{x,1}\geq\delta)\land(I_{x,2}\leq\delta)\right)\geq\mathbb{P}(I_{x,1}\geq\delta)+\mathbb{P}(I_{x,2}\leq\delta)-1
\]
we get

\[
\mathbb{P}(I_{x}^{(m)}\geq0)>1-O\left(\exp\left(-\exp\left(\frac{\ln x}{D_{1}\ln_{2}x}\right)\right)\right)
\]
for some large positive constant $D_{1}.$ This completes the proof
of Theorem \ref{Theorem 1.3}. \qed

\end{document}